\documentclass{article}

\usepackage{amsmath,amsfonts,amssymb,amsthm}
\usepackage{graphicx}
\usepackage{dsfont}
\usepackage[authoryear]{natbib}
\usepackage{enumerate}
\usepackage{mathrsfs}
\RequirePackage[colorlinks,citecolor=blue,urlcolor=blue]{hyperref}

\DeclareMathAlphabet{\mathpzc}{OT1}{pzc}{m}{it}

\newcommand{\desc}{\alpha} 
\newcommand{\R}{\mathds{R}}
\newcommand{\T}{\bf T}

\newcommand{\I}{\mathcal{I}} 
\newcommand{\continuation}{\mathcal{C}}  
\newcommand{\stopping}{\mathcal{S}}        
\newcommand{\negative}{\mathcal{N}}        
\newcommand{\positive}{\mathcal{P}}
\newcommand{\ngi}{N}
\newcommand{\ci}{C}
\newcommand{\arbitraryinterval}{D}
\newcommand{\pairs}{{\Theta}}

\newcommand{\Ga}{G_{\desc}} 
\newcommand{\Ra}{R_{\desc}} 
\newcommand{\Va}{V}
\newcommand{\Wa}{W_{\desc}}
\newcommand{\phia}{\varphi_\desc}
\newcommand{\psia}{\psi_\desc} 
\newcommand{\wa}{w_\desc}
\newcommand{\igen}{\mathcal{L}}
\newcommand{\hit}[1]{\mathpzc{h}_{#1}}
\newcommand{\ind}[1]{\mathds{1}_{#1}}
\newcommand{\D}{\mathcal{D}}
\newcommand{\CC}{\mathscr{C}} 

\newcommand{\ea}[1]{e^{-\desc{#1}}}
 
\def\P{\operatorname{\mathds{P}}}
\newcommand{\Ex}[2]{\E_{#1}\left(#2\right)}
\newcommand{\E}{\operatorname{\mathds{E}}}

\theoremstyle{plain}
\newtheorem{teo}{Theorem}[section]

\newtheorem{lemma}[teo]{Lemma}
\newtheorem{proposition}[teo]{Proposition}
\theoremstyle{definition}
\newtheorem{remark}[teo]{Remark}
\newtheorem{example}[teo]{Example}
\newtheorem{cond}[teo]{Condition}
\newtheorem{algorithm}{Algorithm}[section]

\begin{document}
\title{An algorithm to solve optimal stopping problems for one-dimensional diffusions}
\author{Fabi\'an Crocce\thanks{Centro de Matem\'atica, Facultad de Ciencias, Universidad de la Rep\'ublica, Uruguay
fcrocce@gmail.com}
\qquad
Ernesto Mordecki\thanks{Centro de Matem\'atica, Facultad de Ciencias, Universidad de la Rep\'ublica, Uruguay
mordecki@cmat.edu.uy}}
\maketitle
\begin{abstract}
Considering a real-valued diffusion, 
a real-valued reward function and a positive discount rate, 
we provide an algorithm to solve the optimal stopping problem consisting in finding the optimal expected discounted reward and the optimal stopping time at which it is attained. 
Our approach is based on Dynkin's characterization of the value function.
The combination of Riesz's representation of $\desc$-excessive functions and the inversion formula gives the 
density of the representing measure, being only necessary to determine its support.
This last task is accomplished through an algorithm. 
The proposed method always arrives to the solution, thus no verification is needed,
giving, in particular, the shape of the stopping region. 
Generalizations to diffusions with atoms in the speed measure and to non smooth payoffs are analyzed.
\end{abstract}
\section{Introduction}
Given a diffusion $X=\{X_t\colon t\geq 0\}$ taking values in an interval $\I\subset\R$,  
a non-negative continuous reward function $g\colon\I\to \R$, 
and a discount factor $\alpha>0$, 
consider the optimal stopping problem consisting in finding the \emph{value function} $\Va(x)$ and the 
\emph{optimal stopping rule} $\tau^*$, 
such that
\begin{equation}\label{eq:osp}
\Va(x)=\Ex{x}{\ea{\tau^*} g(X_{\tau^*})} = \sup_{\tau\in\T}\Ex{x}{\ea{\tau} g(X_{\tau})}.
\end{equation}
Here $\T$ is the class of stopping times and we consider $g(X_\tau) = 0$ if $\tau=\infty$
(see section \ref{section:main} for definitions).
Optimal stopping for real-valued diffusions  is a well established and rich area of research.
It can be inscribed into the class of \emph{markovian} stopping problems, 
existing many different approaches to solve them.
One of the most popular ones, 
the \emph{free boundary} approach (see \citet{PeskirShiryaev:2006} with the historical comments in pp. 50-52 and the corresponding references),
when applicable, is very effective and consists in two steps: 
the solution of a free boundary differential equation to find a candidate solution;
and the verification (usually through stochastic calculus) that the candidate is the true solution.
The celebrated \emph{smooth fit condition} becomes a key tool in this framework.
A second approach we mention is 
Dynkin's characterization of the value function \citep{Dynkin:1963}.
It was used for instance by \citet{Taylor:1968}, and has 
a variation proposed in \citet{DayanikKaratzas:2003} (where concavity is used instead of excesiveness).
Other several different approaches can be found in Chapter IV of \citet{PeskirShiryaev:2006}.


Under the conditions assumed in this paper (see section \ref{section:proof}),
the optimal stopping rule exists and has the form
\begin{equation}\label{eq:osr}
\tau^*=\inf\{t\geq 0\colon X_t\in \stopping\},
\end{equation}
where the \emph{stopping region}  is the closed set
\begin{equation}\label{eq:stopping}
\stopping=\{x\in \I\colon V(x)=g(x)\}.
\end{equation}
The \emph{continuation region} is
$\continuation=\I\setminus\stopping$.
A key role in our proposal is played by the \emph{negative set}
$$
\negative=\{x\in\I\colon(\desc-\igen)g(x)<0\}.
$$
where $\igen$ is the infinitesimal generator of $X$.


Regarding applications, it must be noticed that in most of the problems where a solution can be found, 
the continuation region  $\continuation$ is either a half-line, giving one-sided solutions,
or a finite interval, giving rise to a  two-threshold policy (or a two-sided solution). 
The first situation appears typically in perpetual American options (see for instance \citet{mckean} and \citet{Merton:1973}), 
also in the problem of disruption for a Wiener process (see section 4.4  in \citet{Shiryaev:2008}). 
The second one appears in the case of sequential testing of two simple hypotheses of the mean of a Wiener process 
(see section 4.2  in \citet{Shiryaev:2008}),
in the quickest detection problem \citep{Shiryaev:2010}, 
and also in the pricing of a perpetual  straddle or strangle option \citep{GerberShiu}.
The analysis of continuation intervals appears in \cite{Alvarez:2001} under restrictions
on the shape of the continuation region.
More recent references on one-sided and two-sided solutions are
for instance \cite{RuschendorfUrusov:2008} and \cite{Lempa:2010}.
\cite{LambertonZervos:2013} obtain verification results in a framework of weak solutions
of SDE with measurable coefficients and a state dependent discount. 


As we mentioned above, the continuation region of solvable problems are usually half lines or intervals.
More important, the shape of this set should be known in advance in order to solve the problem.
Solved cases with different continuation regions are seldom treated in the literature,
as in order to apply the smooth pasting condition, one has to guess first the structure of this set.
Furthermore, and perhaps more relevant to our discussion, 
examples are usually solved based on verification results 
(see the discussion in the Introduction in \cite{LambertonZervos:2013} with the references therein). 

The method that we propose in the present paper consists in the following steps:
\begin{enumerate}
\item[(1)] 
Apply Dynkin's characterization to obtain that the value function is the minimal excessive function that is a majorant of the reward.
\item[(2)] 
Represent this excessive function as an integral of the Green kernel of the process w.r.t. a representing measure.
\item[(3)] 
Identify the support of this representing measure as the stopping region  $\stopping$, based on the fact that the value function is
harmonic on the continuation region.
\item[(4)] 
Identify the density (w.r.t. the speed measure) of the representing measure through the inversion formula.
\item[(5)] 
Determine the support $\stopping$ of this measure through an algorithm, 
constructing $\continuation=\I\setminus\stopping$ as an enlargement of the set $\negative$.
\end{enumerate}
The proposed method departs from the scale function and speed measure (that determine the generator $\igen$) 
and the increasing and decreasing
solutions of the equation $\alpha u=\igen u$ (that determine the Green kernel),
and gives the complete solution of the problem without need of further verification.
The main restriction of the method is that the negative set should be a finite union of 
intervals, 
i.e. $\negative =\bigcup_{i=1}^n \ngi_i$.
%
It is important to note that $\negative$ is directly computed from the data of the problem.
The steps of the algorithm to construct the set $\continuation$ are the following:
\begin{enumerate}
\item (enlargment) for each $\ngi_i$ construct the largest possible interval ${\ci_i}\supset \ngi_i$ contained in the continuation region (see Condition \ref{cond:continuation});
\item if ${\ci_i}$ are pairwise disjoint intervals then $\continuation=\bigcup_i {\ci_i}$;
\item else (merge), denote by $\ngi_i$ each connected component of  $\bigcup_i {\ci_i}$ and return to step 1.
(Observe that the number of intervals strictly decreases.)
\end{enumerate}

As a consequence, 
the algorithm's output are the connected components of the continuation region 
(whose number is smaller or equal than $n$) 
determining if the problem is one-sided, two-sided, 
or other (i.e. the continuation region  is the union of several intervals).
The value function is then written as an integral of the Green kernel w.r.t. the just obtained measure, 
and this integral gives the classical form of value function as a linear combination of the fundamental solutions in each continuation interval.


The representation of excessive functions in optimal stopping of diffusions was initiated by \cite{Salminen:1985}, 
who represents the value function in terms of the Martin kernel. 
Afterwards, \cite{MordeckiSalminen:2007} use the Green kernel for optimal stopping of Hunt and L\'evy processes. 
The identification of the representing measure through the inversion formula 
was obtained in \cite{Crocce:2014},
and appears in \cite{CrocceMordecki:2012} for one-sided problems, 
and in \cite{ChristensenEtAl:2019} for multidimensional diffusions. 
It can be traced back to formula (8.30) in \cite{Dynkin:1969}, for the cases when
the limit therein can be interchanged with the integral.
More recently, 
also based in representation methods,
disconnected stopping regions where obtained for optimal stopping problems
for diffusions with discontinuous coefficients in \cite{MordeckiSalminen:2019a,MordeckiSalminen:2019b}.


The rest of the paper is as follows. 
In section \ref{section:main} we introduce the necessary definitions and the main result, 
in section \ref{section:generalizations} we discuss possible generalizations in two separate directions: 
diffusions with atoms in the speed measure, and non-smooth (but still continuous) rewards.
Section \ref{section:examples} presents the implementation of the algorithm and contains two examples, and section \ref{section:proof}
contains the proof of the main result.
\section{Main result}\label{section:main}
Consider a conservative and regular one-di\-men\-sion\-al diffusion $X=\{X_t\colon t\geq 0\}$, 
in the sense of \cite{ItoMcKean:1974} (see also \cite{BorodinSalminen:2002}). 
The state space of $X$ is denoted by $\I$, 
an interval of the real line $\R$  with left endpoint $\ell=\inf\I$ and right endpoint $r=\sup\I$,
where $-\infty\leq\ell<r\leq\infty$. The boundaries can be of any kind but killing. 
Denote by $\P_x$ the probability measure associated with $X$ when starting from $x$, 
and by $\E_x$ the corresponding mathematical expectation. 
The set of stopping times $\T$ is considered with respect to the usual augmentation of the natural 
filtration generated by $X$ (see I.14 in \cite{BorodinSalminen:2002}).


Denote by $\igen$ the \emph{infinitesimal generator} of the diffusion $X$, and by $\D(\igen)$ its domain.
For any stopping time $\tau$ and for any $f\in \D(\igen)$ 
 the following  discounted version of the Dynkin's formula holds:
\begin{equation}
\label{eq:dynkinFormula}
 f(x)=\E_x \left(\int_0^{\tau} \ea{t} (\desc-\igen)f(X_t) dt\right)+ \E_x(\ea{\tau}f(X_\tau)).
\end{equation}
The \emph{resolvent} of the process $X$ is the operator $\Ra$ defined by
\begin{equation*}
\Ra u(x)=\int_0^\infty e^{-\alpha t}\E_x u(X_t)dt,
\end{equation*}
applied to a function $u\in \CC_b(\I)=\{u\colon\I\to \R, u \mbox{ is continuous and bounded}\}$. The image of the operator $\Ra$ is independent of  $\alpha>0$ and coincides with the domain of the infinitesimal generator $\D(\igen)$. 
Moreover, for any $f\in \D(\igen)$, $\Ra (\desc-\igen)f = f$, and for any $u\in \CC_b(\I)$, $(\desc-\igen)\Ra u = u$. 
In other terms, $\Ra$ and $\desc-\igen$ are inverse operators (see Prop. VII.1.4 in \cite{RevuzYor:1999}).
Denoting by $s$ and $m$ the scale function and the speed measure of the diffusion $X$ respectively, 
we have that,
for any $f \in \D(\igen)$,
the lateral derivatives with respect to the scale function exist for every $x\in (\ell,r)$.
Furthermore, they satisfy 
\begin{equation}\label{eq:atom} 
\frac{\partial^+ f}{\partial s}(x)- \frac{\partial^- f}{\partial s}(x)=  m(\{x\}) \igen f(x),
\end{equation}
and the following identity holds for $z>y$:
\begin{equation*}
 \frac{\partial^+ f}{\partial s}(z)-\frac{\partial^+ f}{\partial s}(y)=\int_{(y,z]} \igen f(x) m(dx).
\end{equation*}
This last formula allows to compute the infinitesimal generator of $f$ at $x\in(\ell,r)$ 
by the Feller's differential operator \citep{Feller:1957}
\begin{equation}
\label{eq:difop}
\igen f(x)=\frac{\partial}{\partial m} \frac{\partial^+}{\partial s}f(x).
\end{equation}
Given a function $u\colon \I \to \R$, and $x\in (\ell,r)$ we give to $\igen u(x)$ the meaning given in \eqref{eq:difop} if it makes sense. 
We also define $\igen u(\ell)=\lim_{x\to \ell^+}\igen u(x)$, if the limit exists.
There exist two continuous functions $\phia \colon \I \mapsto \R^+$ decreasing,
and $\psia\colon \I\mapsto \R^+$ increasing, solutions of $\alpha u = \igen u$, 
such that any other continuous function $u$ is a solution of the differential equation if and only if 
$u=a\phia+b\psia$, with $a,b$ in $\R$. 
Denoting by $\hit{z}=\inf \{t\colon X_t = z\}$ the hitting time of level $z\in\I$, we have 
\begin{equation}\label{eq:hitting}
\E_x (e^{-\alpha \hit{z}})=
\begin{cases}
	\frac{\psia(x)}{\psia(z)},\quad x\leq z,\\[.5em]
	\frac{\phia(x)}{\phia(z)},\quad x\geq z.
\end{cases}
\end{equation}
The functions  $\phia$ and $\psia$, though not necessarily in $\D(\igen)$, also satisfy \eqref{eq:atom}
for all $x\in (\ell,r)$,
so that in case $m(\{x\})=0$, 
the derivative at $x$ of both functions with respect to the scale  exists. 
The $\alpha$-\emph{Green function} of $X$ is defined by
\begin{equation*}
\Ga(x,y)=\int_0^\infty e^{-\alpha t}p(t;x,y)dt,
\end{equation*}
where $p(t;x,y)$ is the transition density of the diffusion with respect to the speed measure $m(dx)$
(this density always exists, see \cite{BorodinSalminen:2002}). 
The Green function may be expressed in terms of $\phia$ and $\psia$ as follows:
\begin{equation}
\label{eq:Garepr}
G_\alpha(x,y)=
\begin{cases}
w_\alpha ^{-1} \psia(x) \phia (y),\quad  & x\leq y, \\
w_\alpha ^{-1} \psia(y) \phia (x),  & x\geq y,
\end{cases}
\end{equation}
where $w_\alpha$, the \emph{Wronskian}, given by
\begin{equation*}
 w_\alpha=\frac{\partial \psia^+}{\partial s}(x)\phia(x)-\psia(x)\frac{\partial \phia^+}{\partial s}(x),
\end{equation*}
is positive and independent of $x$ (\cite{BorodinSalminen:2002}). 
For general reference on diffusions and Markov processes see 
\cite{BorodinSalminen:2002,ItoMcKean:1974,RevuzYor:1999,Dynkin:1965,KaratzasShreve:1991}.


A  non-negative Borel function $u\colon\I\to \R$ is called \emph{$\alpha$-excessive} for the process $X$ if
$e^{-\alpha t}\E_x(u(X_t))\leq u(x)$ for all $x\in \I$ and $t\geq 0$, and
 $\lim_{t \to 0} \E_x(u(X_t))= u(x)$ for all $x\in \I$.
A 0-excessive function is said to be \emph{excessive}. 
%
%
Dynkin's characterization \citep{Dynkin:1963} states that, if the reward function is lower semi-continuous,  $V$ is the value function of the non-discounted optimal stopping problem with reward $g$ if and only if $V$ is the least excessive function such that $V(x)\geq g(x)$ for all $x\in \I$. 
Applying this result to a killed process \citep{CrocceMordecki:2012}, we obtain that $\Va$, 
the value function of the problem \eqref{eq:osp}, 
is characterized as the least $\alpha$-excessive majorant of $g$ .


A key feature of our proposal is the representation of excessive functions as integrals of the Green kernel.
The Riesz's representation of an $\alpha$-excessive function states that a function $u\colon \I \to \R$ is $\alpha$-excessive if and only if there exist a non-negative Radon measure $\mu$ 
on $[\ell,r]$ such that
\begin{equation}
\label{eq:alphaexcessive}
 u(x)=\int_{(\ell,r)}\Ga(x,y)\mu(dy) + \mu(\{\ell\})\varphi_\alpha(x)+\mu(\{r\})\psi_\alpha(x).
\end{equation}
Furthermore, the previous representation is unique. 
The measure $\mu$ is called the representing measure of $u$.
Formula \eqref{eq:alphaexcessive} is obtained from II.29 in \cite{BorodinSalminen:2002}.


We next formulate our main result in a smooth framework: the value function $g$ 
satisfies the inversion formula (in particular $\igen g(x)$ must be defined for all $x\in\I$):
\begin{equation} \label{eq:inversionformula}
g(x)=\int_{\I} \Ga(x,y)(\desc-\igen)g(y)m(dy);
\end{equation}
and the speed measure has no atoms. 
The proof of this result is deferred to section \ref{section:proof}.
A discussion of possible generalizations in presented in section \ref{section:generalizations}.
Denote
\begin{equation}\label{eq:sigma}
\sigma(dy)=(\desc-\igen)g(y)m(dy).
\end{equation}
\begin{teo} \label{theorem:main}
Consider a diffusion $X$ whose speed measure has no atoms. 
Assume that the reward function $g$ satisfies the inversion formula \eqref{eq:inversionformula}
and that the negative set is a finite union of $n\geq 1$ 
disjoint intervals, i.e.
$$
\negative=\cup_{i=1}^n\ngi_i.
$$
Then, the value function of the OSP is
 \begin{equation}
 \label{eq:VaDiffGeneral}
  \Va(x)=\int_{\stopping} \Ga(x,y) \sigma(dy),
 \end{equation}
 where the continuation region $\continuation=\I\setminus \stopping$ is a finite union of $1\leq m\leq n$ disjoint 
 intervals
 $\ci_i$,
 i.e.
$$
\continuation=\cup_{i=1}^m\ci_i,
$$
s.t. $\negative\subset\continuation$, and 
\begin{enumerate}[\rm(a)]
\item\label{a} if $\ell<\inf C$, then $\int_{\ci_i}\phia(y)\sigma(dy)=0$,
\item\label{b} if $\sup C<r$, then $\int_{\ci_i}\psia(y)\sigma(dy)=0$,
\item \label{c} for $x\in \ci_i$, $\int_{\ci_i}\Ga(x,y)\sigma(dy)\leq 0.$
\end{enumerate}
Furthermore, the continuation region $\continuation$ can be found by Algorithm \ref{algor}, to be presented further on.
\end{teo}
\begin{remark} \label{remark:formaVgeneral}
If $x\in\ci_i=(\ell_i,r_i)$, according to \eqref{eq:VaDiffGeneral}, we have
\begin{multline*}
\Va(x)=\int_{\I\setminus \continuation}\Ga(x,y)\sigma(dy)\\
=\int_{(\I\setminus \continuation)\cap \{x<\ell_i\}}\wa^{-1}\psia(y)\phia(x)\sigma(dy)+\int_{(\I\setminus \continuation)\cap \{x>r_i\}}\wa^{-1}\psia(x)\phia(y)\sigma(dy)\\
=k_1^i \phia(x)+k_2^i \psia(x).
\end{multline*}
Applying the representation Lemma \ref{lem:waeqg} we know that $\Va(\ell_i)=g(\ell_i)$ and $\Va(r_i)=g(r_i)$, 
obtaining
 \begin{equation}\label{eq:continuousfit1}
 k_1^i=\frac{g(r_i)\psia(\ell_i)-g(\ell_i)\psia(r_i)}{\psia(\ell_i)\phia(r_i)-\psia(r_i)\phia(\ell_i)},
 \end{equation}
 and
  \begin{equation}\label{eq:continuousfit2}
   k_2^i=\frac{g(\ell_i)\phia(r_i)-g(r_i)\phia(\ell_i)}{\psia(\ell_i)\phia(r_i)-\psia(r_i)\phia(\ell_i)}.
   \end{equation}
In the particular case in which $\ell_i=\ell$ we have $k_1^i=0$ and $k_2^i=g(r_i)/\psia(r_i)$, and if $r_i=r$ then $k_1=g(\ell_i)/\phia(\ell_i)$ and $k_2=0$. We have then the classical alternative formula
\begin{equation}
\label{eq:VaMasHumana}
\Va(x)=
\begin{cases}
g(x),&\text{for $x \notin \continuation$,}\\
k_1^i\phia(x)+k_2^i\psia(x), &\text{for $x \in \ci_i\colon i=1\ldots m$.}
\end{cases}
\end{equation}
\end{remark}
The coefficients in \eqref{eq:continuousfit1} and \eqref{eq:continuousfit2} 
appeared (in a slightly different form) in \cite{Alvarez:2001}, and also in \cite{Lempa:2010}, and \cite{LambertonZervos:2013}. 


As we mentioned above, 
the continuation region is constructed as an enlargement of the negative set,
and this is done by enlarging each of the 
intervals $\ngi_i$ of $\negative$. 
Introduce $\positive=\I\setminus\negative$ the positive part of the support of $\sigma$, 
and denote by $\sigma^+(dx)$ the measure 
\begin{equation*} 
\sigma^+(dx):= \sigma(dx)\ind{\positive}(x),
\end{equation*}
where $\sigma$ is given in \eqref{eq:sigma},
and,
for an arbitrary interval $\arbitraryinterval\subset\I$ define the signed measure $\sigma_\arbitraryinterval$ by
\begin{equation}\label{eq:arbitrary}
\sigma_\arbitraryinterval(dx)=\ind{\arbitraryinterval}(x)\sigma(dx)+\sigma^+(dx)\ind{\I\setminus\arbitraryinterval}.
\end{equation}
Observe that $\sigma_\arbitraryinterval$ is a positive measure outside $\arbitraryinterval$, 
equal to $\sigma$ in $\arbitraryinterval$.

The following statement specifies what are the conditions that the enlarged interval $\ci$ should satisfy. 
\begin{cond} \label{cond:continuation}
We say that the pair of intervals $(\ngi,\ci) \colon \ngi\subseteq\ci\subseteq \I$ satisfy 
Condition \ref{cond:continuation} if the following assertions hold:
\begin{enumerate}[(i)]
\item \label{i} both, $\int_\ngi \phia(x) \sigma(dx)\leq 0$ and $\int_\ngi \psia(x) \sigma(dx)\leq 0$;
\item \label{ii} if $\ell<\inf \ci$, then $\int_{\ci}\phia(x)\sigma_\ngi(dx)=0$;
\item \label{iii} if $\sup \ci< r$, then $\int_{\ci}\psia(x)\sigma_\ngi(dx)=0$;
\item \label{iv} for every $x\in \ci$, $\int_{\ci}\Ga(x,y)\sigma_\ngi(dy)\leq 0.$
\end{enumerate}
\end{cond}
The algorithm to construct the continuation region follows.
\begin{algorithm} \label{algor}
(Starting from a subset of the continuation region, in subsequent steps, increase the considered subset until finding the actual continuation region.)

\begin{itemize}
\item[BS.]\emph{(base step)} Consider disjoint 
intervals $\ngi_1,\ldots,\ngi_n \subseteq \I$ such that
$$
\negative=\left\{x\in \I:(\desc-\igen)g(x)<0\right\}=\bigcup_{i=1}^n \ngi_i.
$$
Consider for each $i$,  the interval ${\ci_i}$ such that $(\ngi_i,{\ci_i})$ satisfies Condition \ref{cond:continuation} 
(this can be done in virtue of Lemma \ref{lemma:enlargement}). Define 
$$
\pairs=\left\{(\ngi_i,{\ci_i}):i=1\ldots n\right\},
$$
and go to the iterative step (IS) with $\pairs$.

\item[IS.]\emph{(iterative step)} At this step we assume given a set $\pairs$ of pair of intervals satisfying Condition \ref{cond:continuation}. We assume the notation\footnote{we remark that at different moments the algorithm execute this step, the notation refers to different objects, e.g. the set $\pairs$ is not always the same set.}
\begin{equation*}
\pairs=\{(\ngi_i=(a_i,b_i),{\ci_i}=({\bar{a}_i},{\bar{b}_i}))\colon i=1 \ldots n\},
\end{equation*}
with $a_i < a_j$ if $i<j$ (the intervals $N_i$ are ordered) and $b_i<a_{i+1}$ (the intervals $N_i$ are disjoint)

\begin{itemize}

\item If for some $j$, ${\ci_j}=\I$, the algorithm is finished and the continuation region is $\I$.

\item Else, if the intervals ${\ci_i}$ are pairwise disjoint, the algorithm is finished and the continuation region is 
$$\continuation=\bigcup_{i=1}^n {\ci_i}$$

\item Else, if ${\bar{a}_j}=\ell$ for some $j>1$, add to $\pairs$ the pair $(\ngi=(\ell,b_j),\ci)$ satisfying 
Condition \ref{cond:continuation}, and remove from $\pairs$ the pairs $(\ngi_i,{\ci_i})$ for $i=1\ldots j$. Observe that the existence of $\ci$ is proved in 
Lemma \ref{lem:caso1}. Return to the iterative step (IS).

\item Else, if ${\bar{b}_j}=r$ for some $j<n$, add to $\pairs$ the pair $(\ngi=(a_j,r),\ci)$ satisfying Condition \ref{cond:continuation}, and remove from $\pairs$ the pairs $(\ngi_i,{\ci_i})$ for $i=j\ldots n$ (observe that the existence of $\ci$ is proved in 
Lemma \ref{lem:caso2}). Return to the iterative step (IS).

\item Else, if for some $j$, ${\ci_j}\cap {\ci_{j+1}}\neq \emptyset$, remove from $\pairs$ the pairs $j$ and $j+1$, and add to $\pairs$ the pair $(\ngi=(a_j,b_{j+1}),\ci)$ satisfying Condition \ref{cond:continuation} 
(its existence is guaranteed, depending on the situation  by Lemmas 
\ref{lemma:merge}, 
\ref{lem:casoizq}, 
\ref{lem:casoder} or 
\ref{lem:caso2lados}). 
Return to the iterative step (IS).
\end{itemize}
\end{itemize}
\end{algorithm}
Finally note that, each time when we return to the iterative step the number of pairs of intervals in $\pairs$ decreases, 
the algorithm performs at maximum $n$ steps.

\section{Generalizations}\label{section:generalizations}

\subsection{Diffusions with atoms in the speed measure}

The absence of atoms of the speed measure was required only for simplicity of exposition. 
A modification of the main result can be formulated also when the speed measure has a finite number of atoms.
The main difference is that the functions 
$$
z \mapsto\int_{(z,b)}\phia(x)\sigma_\ngi(dx),\quad
z \mapsto\int_{(a,z)}\psia(x)\sigma_\ngi(dx),
$$
in the proof of Lemma \ref{lemma:enlargement} can be discontinuous, having finite jumps at the atoms.
Then, if one of the extremes of an interval happens to be an atom, 
in order to verify \eqref{ii} and \eqref{iii} in Condition \ref{cond:continuation}, 
the representing measure should contain part of the mass of the atom, 
and the smooth fitting does not hold.
This situation, 
with the presentation of corresponding examples, 
was examined in \cite{CrocceMordecki:2012}.

\subsection{More general reward functions}
In many situations the reward function $g$ is not regular enough to satisfy the inversion formula \eqref{eq:inversionformula}.
Assume then that there exists a measure $\nu$ such that
\begin{equation}
\label{eq:mug}
g(x)=\int_\I \Ga(x,y) \nu(dy),
\end{equation}
where $g$ is non-negative and continuous, 
and $\Ga(x,y)$ is defined by \eqref{eq:Garepr}. 
In these cases, considering the second derivative of the difference of two convex functions as a signed measure, it is possible to obtain a ``generalized'' inversion formula useful for our needs 
(see \cite{Dudley:2002} Problems 11 and 12 of Section 6.3).

Just to consider a simple example, 
assume that $X$ is a standard Brownian motion, 
and  consider the function $g\colon \R \to \R$ given by
\begin{equation*}
g(x):=
\begin{cases}
0,&x\leq 0,\\
x,& 0<x<1,\\
2-x,& 1\leq x \leq 2,\\
x-2, & x>2.
\end{cases}
\end{equation*}
In this case, the differential operator is $\igen f=\frac12f''$ when $f$ is in $\D(\igen)$. The inversion formula \eqref{eq:inversionformula} would be
\begin{equation*}
g(x)=\int_{\R}\Ga(x,y)(\desc-\igen)g(y) m(dy)
\end{equation*}
where $m(dy)=2 dy$, so the candidate to be $\nu$ is $(\desc-\igen)g(y)2 dy$. The derivatives of $g$, in the general sense, would be
\begin{equation*}
g'(x)=
\begin{cases}
1,& x<1\\
-1,& 1 < x < 2 \\
1 & x>2
\end{cases}
\end{equation*}
and the second generalized derivative is the measure 
$-2\delta_1(dx)+2\delta_2(dx)$ (where $\delta_a(dx)$ denotes the Dirac's delta measure at the point $x=a$). 
This lead us to consider
\begin{equation}\label{eq:nu}
\nu(dy)= 2\desc g(y) dy +2 \delta_1(dy)-2\delta_2(dy).
\end{equation}
The corresponding computations show that \eqref{eq:mug} holds with $\nu$ in \eqref{eq:nu}.

\begin{teo} \label{teo:general2}
Consider a one-dimensional diffusion $X$ and a non-negative and continuous reward function 
$g\colon{\I} \to \R$ such that  \eqref{eq:mug} holds, 
with $\nu$ a signed measure on $\I$. 
Suppose that $\ci_i\colon i=1,\ldots, m$ ($m$ could be $\infty$) are pairwise disjoint subintervals of $\I$, 
such that 
\begin{enumerate}[\rm(a)]
\item $\int_{\ci_i}\phia(y)\nu(dy)=0$ if there is some $x\in \I$ such that $x<y$ for all $y\in \ci_i$,
\item $\int_{\ci_i}\psia(y)\nu(dy)=0$ if there is some $x\in \I$ such that $x>y$ for all $y\in \ci_i$.
\end{enumerate}
Define $\stopping$ by 
\begin{equation*}
\stopping=\I\setminus \cup_{i=1}^n \ci_i.
\end{equation*}
and $\Va\colon \I \to \R$ by
\begin{equation*}
\Va(x)=\int_{\stopping} \Ga(x,y)\nu(dy). 
\end{equation*}
If $\nu(dy)\geq 0$ in $\stopping$, and $\Va\geq g$ in $\continuation=\cup_{i=1}^m \ci_i$, then
$\Va$ is the value function associated with the OSP, and $\stopping$ is the stopping region.
\end{teo}

\begin{remark} \label{remark:formaVgeneral2} With the same arguments given in Remark \ref{remark:formaVgeneral} we obtain the alternative representation for $\Va$, given in \eqref{eq:VaMasHumana}.
\end{remark}
\begin{proof}
Based on Theorem 3.3.1 in \citet{Shiryaev:2008} we know that $g$ satisfies Dynkin's characterization.
The strategy for the proof is then to verify that $\Va$ is the minimal $\desc$-excessive function 
that dominates the reward function $g$.
By the definition of $\Va$, 
and taking into account that $\nu$ is a non-negative measure in $\stopping$, 
we conclude that $\Va$ is an $\desc$-excessive function. 
Applying Lemma \ref{lem:waeqg} with $\Wa:=\Va$, 
we conclude that $\Va(x)$ and $g(x)$ are equal for $x\in\stopping$, 
which in addition to the hypothesis $\Va(x)\geq g(x)$ for all $x\in \stopping^c$ allow us to conclude that $\Va$ is a majorant of the reward. 
So far, we know
\begin{equation*}
\sup_\tau \Ex{x}{\ea{\tau}g(X_\tau)}\leq \Va(x). 
\end{equation*}
From Lemma \ref{lem:waeqg} --in the first equality-- we get
\begin{align*}
\Va(x)&=\Ex{x}{\ea{\hit{\stopping}}g(X_{\hit{\stopping}})} \leq \sup_\tau \Ex{x}{\ea{\tau}g(X_\tau)}, 
\end{align*}
that proves the other inequality holds as well. From the previous equation we also conclude that $\stopping$ is the stopping region.
\end{proof}

Comparing Theorem \ref{theorem:main} and Theorem \ref{teo:general2}, it should be emphasized that the former gives a characterization of the solution and a method to find it, while the latter is just a verification theorem, which of course, also suggests a method to find the solution. 
However, Theorem \ref{teo:general2} has less restrictive hypothesis and, although we do not include it here, an algorithm to find the continuation region may be developed, at least when the region in which the measure $\nu$ is negative, is a finite union of intervals; in fact, Algorithm \ref{algor} would be a particular case of this algorithm when considering $\nu(dy)=(\desc-\igen)g(y)m(dy)$.

\section{Examples}\label{section:examples}

In order to apply Theorem \ref{theorem:main} it is necessary to check the inversion formula.
This requires essentially two conditions: enough smoothness and a proper behavior at infinity. 
A reasonable behavior at infinite is the following:
\begin{equation} \label{eq:gOverPsi4}
 \lim_{z\uparrow r} \frac{g(z)}{\psia(z)}=\lim_{z\downarrow\ell} \frac{g(z)}{\phia(z)}=0,
\end{equation}
(For other behaviors see Theorem 6.3 in \cite{LambertonZervos:2013}.)
These conditions are useful to verify the inversion formula \eqref{eq:inversionformula}
for a smooth function, 
as stated in the following result.

\begin{proposition} \label{propInversion4}
Assume that  $\I=(\ell,r)$, 
that for $g\colon\I\to \R$ the differential operator is defined for all $x\in\I$,
and that
\begin{equation} \label{eq:alIntegrable4}
 \int_{\I}\Ga(x,y) |(\desc-\igen)g(y)| m(dy)<\infty.
\end{equation}
Take sequences $\ell_n\downarrow\ell$ and $r_n\uparrow r$ s.t. for each $n$ there exists a function $g_n \in \D(\igen)$ such that $g_n(x)=g(x)$ for all $x\in(\ell_{n+1},r_{n+1})$.
Then, if \eqref{eq:gOverPsi4} holds, the inversion formula \eqref{eq:inversionformula} holds true.
\end{proposition}
\begin{proof}
Under the condition \eqref{eq:alIntegrable4}, an application of Fubini's Theorem gives
\begin{equation*} 
\int_{\I} \Ga(x,y) (\desc-\igen)g(y) m(dy)=\Ra (\desc-\igen)g(x). 
\end{equation*}
Let $\tau_n$ be the hitting time of the set $\I\setminus (\ell_n,r_n)$, defined by
\begin{equation*}
\tau_n:=\inf\{t\geq 0\colon X_t \notin (\ell_n,r_n)\}.
\end{equation*}
Consider $x\in (r_n,\ell_n)$. We have $\tau_n=\inf\{\hit{r_n},\hit{\ell_n}\}$.
By the continuity of the paths it can be concluded that $\tau_n \to \infty,\ (n\to\infty)$. 
Applying Dynkin's formula \eqref{eq:dynkinFormula} to $g_n$ and $\tau_n$ we obtain
\begin{equation*}
g_n(x)=\Ex{x}{\int_0^{\tau_n} \ea{t} (\desc-\igen)g_n(X_t) dt}+ \Ex{x}{\ea{\tau_n}g_n(X_{\tau_n})},
\end{equation*}
taking into account that $g_n(x)=g(x)$ and $(\desc-\igen)g(x)=(\desc-\igen)g_n(x)$ for $\ell_{n+1} <x <r_{n+1}$, from the previous equality follows that
\begin{equation}\label{eq:paraconvdominada4}
g(x)=\Ex{x}{\int_0^{\tau_n} \ea{t} (\desc-\igen)g(X_t) dt}+ \Ex{x}{\ea{\tau_n}g(X_{\tau_n}}).
\end{equation}
About the second term on the right-hand side of the previous equation we have
\begin{align*}
\Ex{x}{\ea{\tau_n}g(X_{\tau_n}})
	&=\Ex{x}{\ea{\hit{r_n}} g(X_{\hit{r_n}})\ind{\{\hit{r_n}<\hit{\ell_n}\}}}
	\\&\qquad
	+\Ex{x}{\ea{\hit{\ell_n}} g(X_{\hit{\ell_n}})\ind{\{\hit{\ell_n}<\hit{r_n}\}}}\\
	&\leq \Ex{x}{\ea{\hit{r_n}} g(X_{\hit{r_n}})}+\Ex{x}{\ea{\hit{\ell_n}} g(X_{\hit{\ell_n}})}\\
	&=\psia(x) \frac{g(r_n)}{\psia(r_n)}+\phia(x)\frac{g(\ell_n)}{\phia(\ell_n)},
\end{align*}
by \eqref{eq:hitting},
which taking the limit as $n\to \infty$ vanishes, by \eqref{eq:gOverPsi4}.
Finally, we can apply Fubini's theorem, and dominated convergence theorem to conclude that the limit as $n\to \infty$ of the first term on the right-hand side of \eqref{eq:paraconvdominada4} is
\begin{equation*}
\int_{\I} \Ga(x,y) (\desc-\igen)g(y) m(dy),
\end{equation*}
thus completing the proof.
\end{proof}

\subsection{Implementation}
\label{sec:implementation}
To compute in practice the optimal stopping region, following the Algorithm \ref{algor}, it can be necessary a computational implementation of some parts of the algorithm. In fact, to solve our examples we have implemented a script in R 
\citep{Rsoftware} that receives as input:
\begin{itemize}
\item[-] the function $(\desc-\igen)g$;
\item[-] the density of the speed measure $m$;
\item[-] the atoms of the speed measure $m$;
\item[-] the functions $\phia$ and $\psia$;
\item[-] two numbers $a$, $b$ that are interpreted as the left and right endpoint of an interval $\ngi$
\end{itemize}
and produce as output two numbers $a'\leq a$, $b'\geq b$ such that $(\ngi,(a',b'))$ satisfy Condition \ref{cond:continuation}. It is assumed that the interval $\ngi$ given as input satisfies the necessary conditions to ensure the existence of $\ngi'$.
To compute $a'$ and $b'$ we use a discretization of the given functions and compute the corresponding integrals numerically. We follow the iterative procedure presented in the proof of Lemma \ref{lemma:enlargement}.
Using this script the examples are easily solved following Algorithm \ref{algor}.

\subsection{Brownian motion with polynomial reward}

Theorem \ref{theorem:main} is specially suited for non-monotone reward functions. 
In the following two examples we consider the same process and reward function with to different
discount values: $\alpha=2$ and $\alpha=1.5$. 
It is known that the stopping region increases with the discount (Prop. 1 in \cite{MordeckiSalminen:2019a}).
More interesting, the algorithm \ref{algor} finds no intersection in the first case 
(so it is not necessary to go back to the iterative step)
but finds an intersection in the second case 
(and goes back to the iterative step). 
As a result in the first case the continuation region has three components, 
and in the second case two.
Furthermore, it is clear that for $\alpha$ small enough, the problem is one sided.
\begin{example}[$\alpha$=2] Consider a standard Brownian motion $X$. Consider the reward function $g$ defined by
\begin{equation*}
g(x):=-(x-2)(x-1)x(x+1)(x+2),
\end{equation*}
and the discount factor $\desc=2$. 
To solve the optimal stopping problem \eqref{eq:osp}, by the application of Algorithm \ref{algor}, we start by finding the set $(\desc-\igen)g(x)<0$. 
As the infinitesimal generator is given by $\igen g(x)=g''(x)/2$, after computations, we find that
$$
\negative=\{x\colon (\desc-\igen)g(x)<0\}=\bigcup_{i=1}^3 \ngi_i,
$$
with $\ngi_1\simeq (-2.95,-1.15)$, $\ngi_2\simeq (0,1.15)$ and $\ngi_3 \simeq (2.95,\infty)$. Computing ${\ci_i}$, as is specified in the (base step) of the algorithm in the proof of Theorem \ref{theorem:main}, we find
${\ci_1}\simeq (-3.23,-0.50)$, ${\ci_2}\simeq (-0.36,1.43)$ and ${\ci_3} \simeq (1.78,\infty)$. 
Observing that these intervals are disjoint we conclude that the continuation region is given by 
${\ci_1}\cup {\ci_2} \cup {\ci_3}$. 
Now, by the application of equation \eqref{eq:VaMasHumana}, we find the value function, which is shown in Fig. \ref{fig:poly-alpha2}. Note that the smooth fit principle holds in the five contact point.
\begin{figure}
 \begin{center} 
\includegraphics[scale=.6]{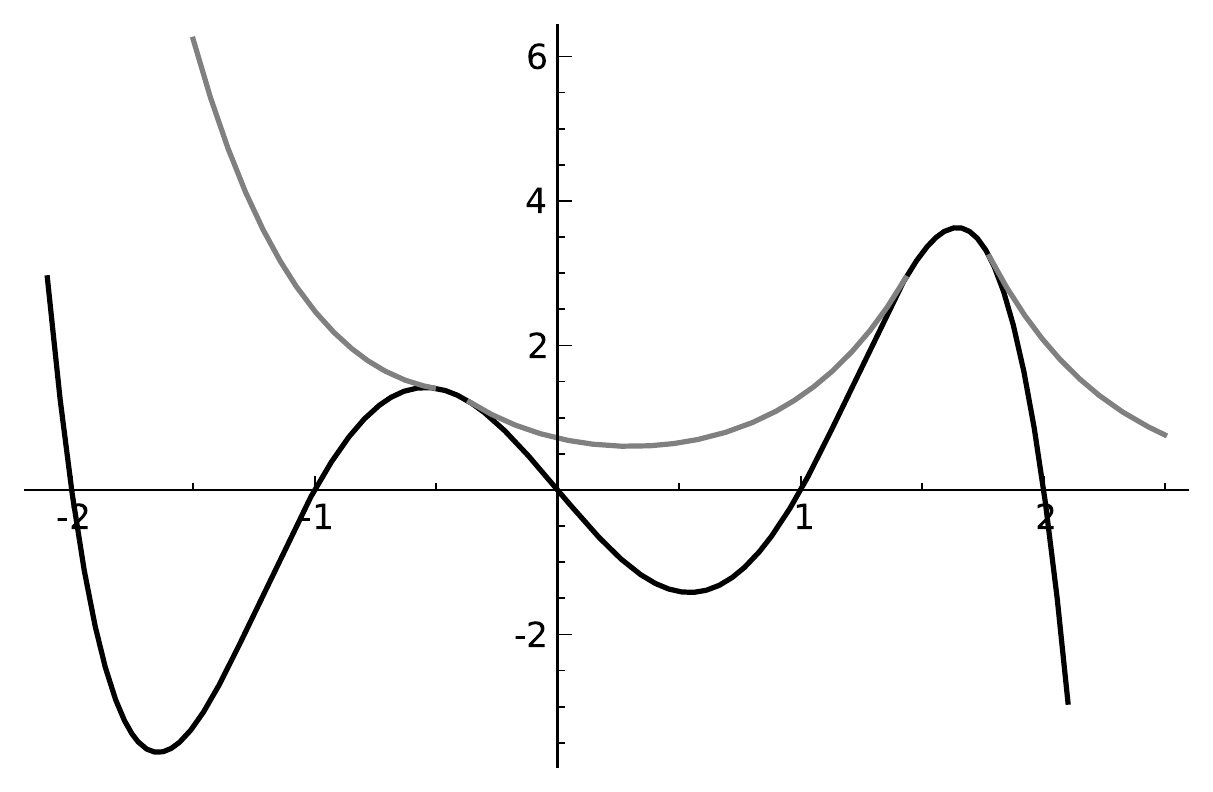}
\caption{\label{fig:poly-alpha2} OSP for the standard BM and a 5th. degree polynomial: $g$ (black), $\Va$ (gray, when different from $g$). Parameter $\desc=2$. }
 \end{center}
\end{figure}
\end{example}
\begin{example}[$\alpha$=1.5]
Consider the process and the reward as in the previous example but with a slightly smaller discount, $\desc=1.5$. 
We have again
\[\{x\colon (\desc-\igen)g(x)<0\}=\bigcup_{i=1}^3 \ngi_i,
\]
but with $\ngi_1\simeq (-3.21,-1.17)$, $\ngi_2\simeq (0,1.17)$ and $\ngi_3 \simeq (3.21,\infty)$. 
Computing ${\ci_i}$ we obtain ${\ci_1}\simeq (-3.53,-0.31)$, ${\ci_2}\simeq (-0.39,1.46)$ 
and ${\ci_3} \simeq (1.76,\infty)$. In this case ${\ci_1}\cap {\ci_2}\neq \emptyset$, therefore, 
according to the algorithm, we have to consider $\ngi_1\simeq (-3.21,1.17)$, obtaining ${\ci_1}\simeq (-3.53,1.46)$. 
Now we have two disjoint intervals and the algorithm is completed. 
The continuation region, shown in Fig. \ref{fig:poly-alpha1-5}, is 
\[ 
\continuation\simeq (-3.53,1.46)\cup (1.76,\infty).
\]
\begin{figure}
 \begin{center} 
\includegraphics[scale=.6]{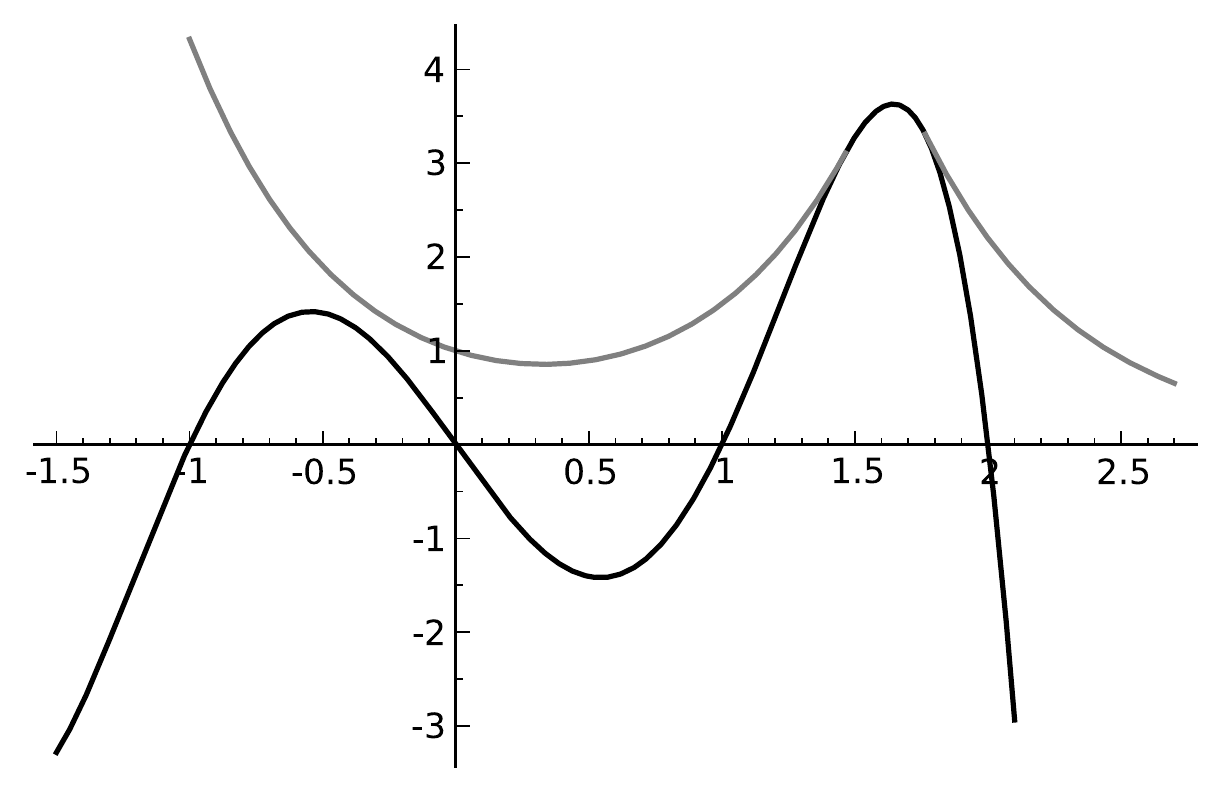}
\caption{\label{fig:poly-alpha1-5} OSP for the standard BM and a 5th. degree polynomial: $g$ (black), $\Va$ (gray, when different from $g$). Parameter $\desc=1.5$. }
\end{center}
\end{figure}
\end{example}
\subsection{Example: A non-differentiable reward}
Consider the OSP with reward $g\colon \R \to \R$ given by 
\begin{equation*}
g(x)=
\begin{cases}
x,& x<1,\\
-x+2,& 1\leq x \leq 2, \\
x-2 & x>2.
\end{cases}
\end{equation*}
This is the function already presented above. 
It satisfies \eqref{eq:mug} with $\nu$ given by \eqref{eq:nu}.
Consider the discount factor $\desc=1$. The measure $\nu$ is negative in $(-\infty,0)$ and in $\{2\}$. Computing exactly in the first case, and by numerical approximation in the second (by following a variant of Algorithm \ref{algor}), we manage to find two disjoint intervals $\ngi_1\simeq(-\infty,1/\sqrt2)$ and $\ngi_2\simeq(1.15,2.85)$ that satisfy the conditions of Theorem \ref{teo:general2}. 
For $\Va$, we have  the expression given in Remark \ref{remark:formaVgeneral2}, which considering $\psia(x)=e^{\sqrt{2\desc}x}$ and $\phia(x)=e^{-\sqrt{2\desc}x}$ in the particular case $\desc=1$, renders\footnote{we approximate the roots.}
\begin{equation*}
V_1(x)=
\begin{cases}
k_2^1 e^{\sqrt2 x},& x<\frac1{\sqrt2},\\
x,& \frac1{\sqrt2} \leq x \leq 1,\\
-x+2,& 1<x\leq 1.15,\\
k_1^2 e^{-\sqrt2 x} + k_2^2 e^{\sqrt2 x},& 1.15<x<2.85,\\
x-2,&x\geq 2.85;
\end{cases}
\end{equation*} 
with $k_2^1= \frac1{e\sqrt2}\simeq 0.26$, $k_1^2\simeq 3.96$ and $k_2^2\simeq 0.013$.
\begin{figure}
\begin{center}
\includegraphics[scale=.6]{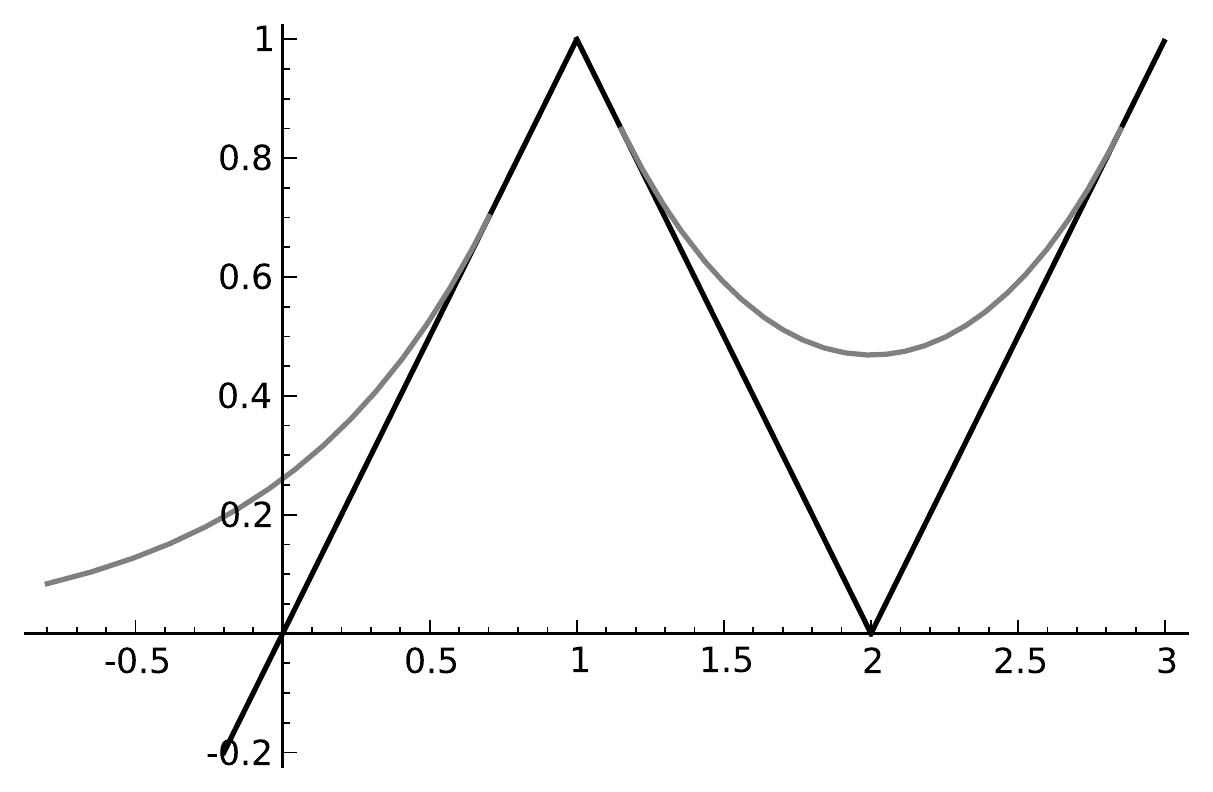}
\end{center}
\caption{\label{fig:mug} OSP for the standard BM and irregular reward: $g$ (black), $V_1$ (gray, when different from $g$).}
\end{figure}
In Figure \ref{fig:mug} we show the reward function $g$ and the value function $V_1$.

\section{Proof of the main result}\label{section:proof}

To begin with the proof, we first observe that 
for a diffusions defined as in section \ref{section:main}, 
excessive functions are continuous (see 29 in \cite{BorodinSalminen:2002}), 
and for the OSP in \eqref{eq:osp},
Theorem 6 pp.~137 in \cite{Shiryaev:2008} is applicable, 
giving that the optimal stopping rule exists and has the form \eqref{eq:osr} with stopping set \eqref{eq:stopping}.
In consequence, as both $g$ and $V$ are continuous functions, the set $\stopping$ is closed.
We follow by presenting a few preliminary results.
\begin{proposition}[Harmonicity]
\label{lem:Wa}
{\rm(a)} Consider $x\in[a,b]\subset\I$ and  
\begin{equation*}
\hit{ab}:=\inf\{t\colon X_t=a\}\wedge\inf\{t\colon X_t=b\}.
\end{equation*}
Then, if $a>\ell$
\begin{equation*}
\phia(x)=\Ex{x}{\ea{\hit{ab}}\phia(X_{\hit{ab}})},
\end{equation*}
and,  if $b<r$ 
\begin{equation*}
\psia(x)=\Ex{x}{\ea{\hit{ab}}\psia(X_{\hit{ab}})}.
\end{equation*}
{\rm(b)} Consider the function $\Wa\colon\I \to \R$ such that
\begin{equation*}
\Wa(x)=\int_S \Ga(x,y)\sigma(dy),
\end{equation*}
where $\sigma$ is a postive measure and the set $S$ is
\begin{equation*}
S=\I \setminus \cup_{i=1}^n \ngi_i,
\end{equation*} 
where $n$ could be infinite, and $\ngi_i$ are disjoint intervals included in $\I$.
Then, the function $\Wa$ satisfies
\begin{equation*}
\Wa(x)=\E_x \left(\ea{\hit{S}} \Wa(X_{\hit{S}})\right).
\end{equation*}
\end{proposition}
\begin{proof}
(a) Let us proof the first statement, which is a direct consequence of the discounted Dynkin's formula for functions that belong to $\D(\igen)$. As $\phia\notin \D(\igen)$, we consider a function $h\in \CC_b(\I)$ such that $h(x)=0$ for $x\geq a$ and $h(x)>0$ for $x<a$.
Then $f$ defined by $f(x):=\left(\Ra h \right)(x)$ belongs to $\D(\igen)$ and there exist a constant $k>0$ such that for  $x\geq a$, $f(x)=k \phia(x)$ 
\citep[see][section 4.6]{ItoMcKean:1974}. 
The discounted Dynkin's formula holds for $f$, so, for $x\geq a$,
\begin{equation*}
f(x)-\Ex{x}{\ea{\hit{ab}}f(X_{\hit{ab}})}=\Ex{x}{\int_0^{\hit{ab}}(\desc-\igen)f(X_t) dt}.
\end{equation*}
From the continuity of the paths, for $t\in[0,\hit{ab}]$, $X_t\geq a$ and $(\desc-\igen)f(X_t)=h(X_t)=0$, so the right-hand side of the previous equation vanishes. Finally taking into account the relation between $f$ and $\phia$ the conclusion follows.
The second statement is proved in an analogous way.
\par\noindent(b)
If $x\in S$ the result is trivial, because $\hit{S}\equiv 0$. Let us consider the case $x\notin S$. In this case $x \in \ngi_i$ for some $i$; we move on to prove that
\begin{equation*}
\Ga(x,y) = \Ex{x}{\ea{\hit{S}}\Ga(X_{\hit{S}},y)}
\end{equation*} 
for all $y$ in $S$. To see this, let us denote by $a=\inf \ngi_i$ and $b=\sup \ngi_i$, and observe that $\hit{S}=\hit{ab}$. If $b < r$ and $y \geq b$ we have
$\Ga(x,y)=\wa^{-1}\psia(x)\phia(y)$ 
and by (a) we get
\begin{align*}
\Ga(x,y)&= \wa^{-1} \Ex{x}{\ea{\hit{ab}}\psia(X_{\hit{ab}})}\phia(y)\\
&= \Ex{x}{\ea{\hit{ab}}\Ga(X_{\hit{ab}},y)},
\end{align*}
where in the second equality we have used again \eqref{eq:Garepr} and the fact that $\hit{ab}\leq y$. In the case $y\leq a$ we have to do the analogous computation.
Now we can write
\begin{align*}
\Wa(x)&=\int_S \Ga(x,y)\sigma(dy)
=\int_S \Ex{x}{\ea{\hit{S}}\Ga(X_{\hit{S}},y)} \sigma(dy)\\
&=\Ex{x}{\ea{\hit{S}}\int_S \Ga(X_{\hit{S}},y) \sigma(dy)}
=\Ex{x}{\ea{\hit{S}}\Wa(X_{\hit{S}},y)},
\end{align*}
and the result follows.
\end{proof}

\begin{lemma}[Representation]
\label{lem:waeqg}
Let $X$ be a one-dimensional diffusion and consider the function $g\colon \I \to \R$ defined by
\begin{equation*}
g(x):=\int_{\I} \Ga(x,y) \sigma(dy),
\end{equation*}
where $\sigma$ is a signed measure on $\I$. Consider the
 function $\Wa\colon {\I} \to \R$ defined by 
\begin{equation}
\label{eq:defWa2}
\Wa(x):=\int_{S} \Ga(x,y) \sigma(dy),
\end{equation}
where the set $S$ is
\begin{equation*}
S:=\I \setminus \cup_{i=1}^m \ci_i,
\end{equation*} 
where $m$ could be infinite, 
the intervals $\ci_i\subset \I$ are pairwise disjoint, 
and
\begin{enumerate}[\rm(a)]
\item $\int_{\ci_i}\phia(y)\sigma(dy)=0$ if there is some $x\in \I$ such that $x<y$ for all $y\in \ci_i$,
\item $\int_{\ci_i}\psia(y)\sigma(dy)=0$ if there is some $x\in \I$ such that $x>y$ for all $y\in \ci_i$,
\end{enumerate}

Then $g(x)=\Wa(x)$ for all $x\in S$. 
\end{lemma}

\begin{proof}
From the definitions of $g$ and $\Wa$ we get
\begin{align*}
g(x)&=\int_{\I}\Ga(x,y)\sigma(dy)\\
&=\Wa(x) + \sum_{i=1}^m \int_{\ci_i}\Ga(x,y)\sigma(dy). \notag
\end{align*}
To prove the result it is enough to verify that if $x\in S$, then 
\begin{equation*}
\int_{\ci_i}\Ga(x,y)\sigma(dy)=0,\quad \mbox{for all $i$.}
\end{equation*}
Consider $x\in S$, then for any $i=1\ldots n$, we have that $x\notin \ci_i$. Since $\ci_i$ is an interval either $x<y$ for all $y$ in $\ci_i$ or $x>y$ for all $y$ in $\ci_i$. Suppose the first case, from \eqref{eq:Garepr} we obtain
\begin{align*}
\int_{\ci_i}\Ga(x,y)\sigma(dy)=\wa^{-1}\psia(x)\int_{\ci_i}\phia(y)\sigma(dy)=0,
\end{align*}
where the second equality follows from hypothesis. The other case is analogous.
\end{proof}


\subsection{Enlargment}
\begin{lemma}[Enlargment]
\label{lemma:enlargement}
 Under the assumptions of Theorem \ref{theorem:main}, consider an 
 interval $\ngi\subseteq \I$, such that $\sigma(dx)<0$ for $x\in \ngi$.
 Then, there exists an interval $\ci$ such that $(\ngi,\ci)$ satisfies Condition \ref{cond:continuation}
\end{lemma}
\begin{proof}
 Consider $\ngi$ to be $(a,b)$. Assertion (\ref{i}) in Condition \ref{cond:continuation} is clearly fulfilled. Without loss of generality (denoting by $\phia$ the result of multiplying $\phia$ by the necessary positive constant) we may assume 
 \begin{equation*} \int_\ngi \psia(x)\sigma(dx)=\int_\ngi \phia(x)\sigma(dx)<0. \end{equation*}
 Under this assumption, $\phia(a)<\psia(a)$ and $\phia(b)>\psia(b)$. Consider 
 \begin{equation*}
 x_1:=\inf\left\{z\in [\ell,a]\colon \int_{(z,b)}\phia(x)\sigma_\ngi(dx)<0\right\}.
 \end{equation*}
 Since $\phia(x)>\psia(x)$ for $x\leq a$ and $\sigma_\ngi(dx)$ is non-negative in the same region we conclude that $\int_{(x_1,b)}\psia(x)\sigma_\ngi(dx)\leq 0$. Consider $y_1>b$ defined by
 \begin{equation*}
 y_1:=\sup\left\{z\in [b,r]\colon \int_{(x_1,z)}\psia(x)\sigma_\ngi(dx)<0\right\}. 
 \end{equation*}
 Now we consider $x_2\geq x_1$ as 
  \begin{equation*}
 x_2:=\inf\left\{z\in [\ell,a]\colon \int_{(z,y_1)}\phia(x)\sigma_\ngi(dx)<0\right\}
  \end{equation*}
 and $y_2\geq y_1$ as
  \begin{equation*}
 y_2:=\sup\left\{z\in [b,r]\colon \int_{(x_2,z)}\psia(x)\sigma_\ngi(dx)<0\right\}. 
  \end{equation*}
 Following in the same way we obtain two non-decreasing sequences $\ell\leq \{x_n\}\leq a$ and $b\leq \{y_n\}\leq r$. By construction, the interval $\ci=(\lim x_n, \lim y_n)$ satisfies (\ref{ii}) and (\ref{iii}) in Condition \ref{cond:continuation}. To prove \eqref{iv}, first we find $k_1(x)$ and $k_2(x)$ such that
 $$
 \begin{cases}
 k_1(x) \psia(a)+k_2(x) \phia(a)=\Ga(x,a) \\
  k_1(x) \psia(b)+k_2(x) \phia(b)=\Ga(x,b). 
 \end{cases}
 $$
 Solving the system we obtain 
 $$k_1(x)=\frac{\Ga(x,b)\phia(a)-\Ga(x,a)\phia(b)}{\psia(b)\phia(a)-\psia(a)\phia(b)}$$
 and
 $$k_2(x)=\frac{\Ga(x,a)\psia(b)-\Ga(x,b)\psia(a)}{\psia(b)\phia(a)-\psia(a)\phia(b)}.$$
Let us see that $k_1(x),k_2(x)\geq 0$ for any $x\in \ci$: using the explicit formula for $\Ga$ it follows that
\begin{equation*}
k_1(x) =
\begin{cases}
 0 &\mbox{for $x\leq a$,} \\
 \wa^{-1}\phia(b)\frac{\psia(x)\phia(a)-\psia(a)\phia(x)}{\psia(b)\phia(a)-\psia(a)\phia(b)} &\mbox{for $x\in (a,b)$,}\\
 \wa^{-1}\psia(x) &\mbox{for $x\geq b$}.
\end{cases}
\end{equation*}
When $x\in (a,b)$ the numerator and denominator are non-negative because $\phia$ is decreasing and $\psia$ increasing. 
The case of $k_2$ is analogous.
 
 Considering $h(x,y)=k_1(x)\psia(y)+k_2(x)\phia(y)$, it can be seen (discussing for the different positions of $x$ and $y$ with respect to $a$ and $b$) that for all $x\in \ci$, $h(x,y)\leq \Ga(x,y)$ for $y \in (a,b)$ and $h(x,y)\geq \Ga(x,y)$ for $y\notin (a,b)$. From these inequalities we conclude that
 \begin{equation*}
 \int_{\ci}\Ga(x,y)\sigma_\ngi(dy) \leq \int_{\ci}h(x,y)\sigma_\ngi(dy) \leq 0;
 \end{equation*}
 where the first inequality is consequence of $\sigma_\ngi(dy)\geq 0$ in $\I\setminus \ngi$ and $\sigma_\ngi(dy)\leq 0$ in $\ngi$; and the second one is obtained fixing $x$ and observing that $h(x,y)$ is a linear combination of $\psia$ and $\phia$ with non-negative coefficients.
\end{proof}
\begin{lemma}[Left Enlargment] \label{lem:caso1}
Under the assumptions of Theorem \ref{theorem:main}, consider the interval $\ngi=(a,b)$ and $\ci=(\ell,\bar{b})$ (with $\bar{b}<r$) such that $(\ngi,\ci)$ satisfy Condition \ref{cond:continuation}. 
Then, there exists $b'\geq \bar{b}$ such that $(\ngi'=(\ell,b),\ci'=(\ell,b'))$ satisfy Condition \ref{cond:continuation}.
\end{lemma}
\begin{proof}
First observe that, if $\arbitraryinterval\subset\arbitraryinterval'$, based on definition \eqref{eq:arbitrary}, we have that
\begin{equation}\label{eq:negativemeasure}
\sigma_{\arbitraryinterval'}(dx)-\sigma_\arbitraryinterval(dx)\text{ is a negative measure.}
\end{equation}
By hypothesis we know
\begin{equation*}
\int_{\ci}\psia(y)\sigma_{\ngi}(dy)=0.
\end{equation*}
It follows from \eqref{eq:negativemeasure} that
\begin{equation*}
\int_{\ci}\psia(y)\sigma_{\ngi'}(dy)\leq 0.
\end{equation*}
Consider $b'=\sup\{x \in [\bar{b},r) \colon \int_{(\ell,x)}\psia(y)\sigma_{\ngi'}(dy)\leq 0\}$. 
Let us check that 
$\ngi'=(\ell,b),\ci'=(\ell,b')$ satisfies Condition \ref{cond:continuation}.
It is clear that
$$\int_{\ci'}\psia(y)\sigma_{\ngi'}(dy)\leq 0,$$
with equality if $b'<r$. This proves (\ref{iii}) in Condition \ref{cond:continuation}. 
Observe that (\ref{ii}) is automatic, as $\ell=\inf\ci'$. Now we prove \eqref{iv}. Consider
\begin{align} \label{eq:splitIntegral1}
\int_{\ci'}\Ga(x,y)\sigma_{\ngi'}(dy)&=\int_{\ci}\Ga(x,y)\sigma_{\ngi}(dy)+\int_{\ci}\Ga(x,y)(\sigma_{\ngi'}-\sigma_{\ngi})(dy) \notag \\
&\qquad +\int_{\ci'\setminus \ci}\Ga(x,y)\sigma_{\ngi'}(dy).
\end{align}
The first term on the right-hand side is non-positive by hypothesis. 
Let us analyze the sum of the remainder terms. 
Considering the previous decomposition with $\psia(y)$ instead of $\Ga(x,y)$, and taking $\int_{\ci}\psia(y)\sigma_{\ngi}(dy)= 0$ into account, we obtain
\begin{equation}
\label{eq:intdiferpsi}
\int_{\ci}\psia(y)(\sigma_{\ngi'}-\sigma_{\ngi})(dy)+\int_{\ci'\setminus \ci}\psia(y)\sigma_{\ngi'}(dy)\leq 0.
\end{equation}
Consider $k(x)$ such that $k(x)\psia(\bar{b})=\Ga(x,\bar{b})$; we have $k(x)\psia(y)\leq\Ga(x,y)$ if $y\leq \bar{b}$ and $k(x) \psia(y)\geq\Ga(x,y)$ if $y\geq \bar{b}$. 
Also note that $(\sigma_{\ngi'}-\sigma_{\ngi})(dy)$ is non-positive in $\ci$ and $\sigma_{\ngi'}$ is non-negative in $\ci'\setminus \ci$. We get
\begin{multline}\label{eq:ga}
\int_{\ci}\Ga(x,y)(\sigma_{\ngi'}-\sigma_{\ngi})(dy) +\int_{\ci'\setminus \ci}\Ga(x,y)\sigma_{\ngi'}(dy) \\ 
\qquad \leq 
k(x)\left(
\int_{\ci}\psia(y)(\sigma_{\ngi'}-\sigma_{\ngi})(dy)+\int_{\ci'\setminus \ci}\psia(y)\sigma_{\ngi'}(dy)
\right)
\leq 0.
\end{multline}
This completes the proof of \eqref{iv}. 
To prove (\ref{i}), first observe that
\begin{equation}\label{eq:previous}
\int_{\ngi'}\psia\sigma(dy)\leq \int_{\ngi'}\psia\sigma_{\ngi'}(dy)\leq \int_{\ci'}\psia\sigma_{\ngi'}(dy)=0.
\end{equation}
To complete the proof, applying the same arguments in \eqref{eq:previous} to the decreasing solution $\phia$, it is enough to see that
$$
\int_{\ci'}\phia(y)\sigma_{\ngi'}(dy)\leq 0,
$$
Now
\begin{multline*}
\int_{\ci'}\phia(x)\sigma_{\ngi'}(dy)\\=\int_{\ci}\phia(x)\sigma_{\ngi}(dy)+\int_{\ci}\phia(x)(\sigma_{\ngi'}-\sigma_{\ngi})(dy)
 +\int_{\ci'\setminus \ci}\phia(x)\sigma_{\ngi'}(dy)
\\ \leq \int_{\ci}\phia(x)(\sigma_{\ngi'}-\sigma_{\ngi})(dy)
 +\int_{\ci'\setminus \ci}\phia(x)\sigma_{\ngi'}(dy)
\\ \leq k\left(\int_{\ci}\psia(x)(\sigma_{\ngi'}-\sigma_{\ngi})(dy)
 +\int_{\ci'\setminus \ci}\psia(x)\sigma_{\ngi'}(dy)\right)\leq 0.
\end{multline*}
The last equality is \eqref{eq:intdiferpsi}.
The first inequality is a consequence of the hypotesis.
And in the second, $k>0$ is such that  $\phia(\bar{b})=k \psia(\bar{b})$ and the same arguments as in \eqref{eq:ga} 
apply. This concludes the proof.
\end{proof}
\begin{lemma}[Right enlargment]\label{lem:caso2}
Under the assumptions of Theorem \ref{theorem:main}, consider the interval $\ngi=(a,b)$ and $\ci=(\bar{a},r)$ (with $\bar{a}>\ell$), such that $(\ngi,\ci)$ satisfies Condition \ref{cond:continuation}. Then, there exists $a'\leq \bar{a}$ such that $(\ngi'=(a,r),\ci'=(a',r))$ satisfies Condition \ref{cond:continuation}.
\end{lemma}
\begin{proof}
Analogous to the proof of the previous lemma.
\end{proof}

\subsection{Merge}
\begin{lemma}[Merge] \label{lemma:merge}
Under the assumptions of Theorem \ref{theorem:main}, consider $\ngi_1=(a_1,b_1)$, $\ngi_2=(a_2,b_2)$ such that $b_1<a_2$ and $(\desc-\igen)g(x)\geq 0$ for $x$ in $(b_1,a_2)$. Let ${\ci_1}=(\bar{a}_1,{\bar{b}_1})$ and ${\ci_2}=({\bar{a}_2},{\bar{b}_2})$ be intervals such that ${\bar{a}_1}>\ell$, ${\bar{b}_1}<r$, ${\bar{a}_2}>\ell$, ${\bar{b}_2}<r$. Suppose that the two pairs of intervals $(\ngi_1,{\ci_1})$, $(\ngi_2,{\ci_2})$ satisfy Condition \ref{cond:continuation}. 
Then, if ${\ci_1}\cap {\ci_2}\neq \emptyset$, 
considering $\ngi=(a_1,b_2)$, there exists an interval $\ci$ such that $(\ngi,\ci)$ satisfies Condition \ref{cond:continuation}.
\end{lemma}
\begin{proof}
By hypothesis
\begin{equation*}
\int_{{\ci_i}}\phia(x)\sigma_{\ngi_i}(dx)=\int_{{\ci_i}}\psia(x)\sigma_{\ngi_i}(dx)=0,\quad i=1,2.
\end{equation*}
Then
\begin{equation*}
\int_{{\ci_1}\cup {\ci_2}}\phia(x)\sigma(dx)=-\int_{{\ci_1}\cap {\ci_2}}\phia(x)\sigma^+(dx)
\end{equation*}
and
$$\int_{{\ci_1}\cup {\ci_2}}\psia(x)\sigma(dx)=-\int_{{\ci_1}\cap {\ci_2}}\psia(x)\sigma^+(dx).$$
We assume, without loss of generality, that 
$$\int_{{\ci_1}\cap {\ci_2}}\phia(x)\sigma^+(dx)=\int_{{\ci_1}\cap {\ci_2}}\psia(x)\sigma^+(dx)>0$$
and therefore, denoting by $(a',b')$ the interval ${\ci_1}\cup{\ci_2}$, we get:
\begin{equation*}
\int_{(a',b')}\phia(x)\sigma(dx)=\int_{(a',b')}\psia(x)\sigma(dx)<0;
\end{equation*}
$\psia(a')\leq\phia(a')$; and $\psia(b')\geq \phia(b')$. The same procedure in the proof of Lemma \ref{lemma:enlargement}, allow us to construct an interval $\ci$ such that $(\ngi,\ci)$ satisfy (\ref{i}), (\ref{ii}) and (\ref{iii}) in Condition \ref{cond:continuation}. Let us prove \eqref{iv}:
If $x<a_1$ we have $\Ga(x,y)=\wa^{-1}\psia(x)\phia(y)$ for $y\geq a_1$ and $\Ga(x,y)\leq \wa^{-1}\psia(x)\phia(y)$ for $y\leq a_1$; since $\sigma_{\ngi}(dy)$ is non-negative in $y\leq a_1$ we find
\begin{equation*}
\int_{\ci} \Ga(x,y)\sigma_{\ngi}(dy)\leq \wa^{-1}\psia(x) \int_{\ci} \phia(y)\sigma_{\ngi}(dy)\leq 0.
\end{equation*}
An analogous argument prove the assertion in the case $x>b_2$. Now consider $x\in \ngi$, suppose $x<\min\{a_2,{\bar{b}_1}\}$ (in case $x>\max\{b_1,{\bar{a}_2}\}$ an analogous argument is valid), we get 
\begin{align*}
\int_{\ci}\Ga(x,y)\sigma_{\ngi}(dy)&=\int_{{\ci_1}}\Ga(x,y)\sigma_{\ngi_1}(dy)+\int_{{\ci_1}}\Ga(x,y)(\sigma_{\ngi}-\sigma_{\ngi_1})(dy)\\
&\quad+\int_{\ci\setminus {\ci_1}}\Ga(x,y)\sigma_{\ngi}(dy),
\end{align*}
where $\int_{{\ci_1}}\Ga(x,y)\sigma_{\ngi_1}(dy)\leq 0$ by hypothesis. 
We move on to prove that the sum of the second and the third terms on the right-hand side of the previous equation are non-positive, thus completing the proof: Observe that 
$$\Ga(x,y)\leq \wa^{-1}\psia(x)\phia(y)$$
and 
$$\Ga(x,y)=\wa^{-1}\psia(x)\phia(y)\quad  (y\geq \min\{a_2,{\bar{b}_1}\})$$ 
The measure $(\sigma_{\ngi}-\sigma_{\ngi_1})$ has support in $\ngi_2$, where the previous equality holds. The measure $\sigma_{\ngi}(dy)$ is positive for $y<a_1$ where we do not have the equality, then
\begin{align*}
&\int_{{\ci_1}}\Ga(x,y)(\sigma_{\ngi}-\sigma_{\ngi_1})(dy) +\int_{\ci\setminus {\ci_1}}\Ga(x,y)\sigma_{\ngi}(dy)\\
&\leq \wa^{-1}\psia(x)\left(\int_{{\ci_1}}\phia(y)(\sigma_{\ngi}-\sigma_{\ngi_1})(dy) +\int_{\ci\setminus {\ci_1}}\phia(y)\sigma_{\ngi}(dy)\right)\leq 0,
\end{align*}
where the last inequality is a consequence of
\begin{align*}
\int_{\ci}\phia(y)\sigma_{\ngi}(dy)&=\int_{{\ci_1}}\phia(y)\sigma_{\ngi_1}(dy)+\int_{{\ci_1}}\phia(y)(\sigma_{\ngi}-\sigma_{\ngi_1})(dy)\\
&\quad+\int_{\ci\setminus {\ci_1}}\phia(y)\sigma_{\ngi}(dy)\leq 0,
\end{align*}
and 
$$\int_{{\ci_1}}\phia(y)\sigma_{\ngi_1}(dy)=0.$$
This completes the proof.
\end{proof}
\begin{lemma}[Left merge]\label{lem:casoizq}
Under the assumptions of Theorem \ref{theorem:main}, consider $\ngi_1=(\ell,b_1)$, $\ngi_2=(a_2,b_2)$ such that: $b_1<a_2$; and $(\desc-\igen)g(x)\geq 0$ for $x$ in $(b_1,a_2)$. Let ${\ci_1}=(\ell,{\bar{b}_1})$ and ${\ci_2}=({\bar{a}_2},{\bar{b}_2})$ be intervals such that: ${\bar{b}_1}<r$; ${\bar{a}_2}>\ell$; and ${\bar{b}_2}<r$. Suppose that the two pairs of intervals $(\ngi_1,{\ci_1})$, $(\ngi_2,{\ci_2})$ satisfy Condition \ref{cond:continuation}. 
If ${\ci_1}\cap {\ci_2}\neq \emptyset$ then, considering $\ngi=(\ell,b_2)$, there exists $\bar{b}$ such that $(\ngi,\ci=(\ell,\bar{b}))$ satisfies Condition \ref{cond:continuation}.
\end{lemma}
\begin{proof}
Define $\bar{b}=\sup\{x \in [{\bar{b}_2},r) \colon \int_{(\ell,x)}\psia(y)\sigma_{\ngi}(dy)\leq 0\}$ (note that $\bar{b}_2$ belongs to the set). We have
\begin{equation}
\label{eq:intPsiaMenor0}
\int_{\ci}\psia(y)\sigma_{\ngi}(dy)\leq 0,
\end{equation}
with equality if $\bar{b}<r$, proving (\ref{ii}) in Condition \ref{cond:continuation}. 
To prove \eqref{iv} we split the integral as follows:
\begin{align}
\label{eq:splitintegralGa}
\int_{\ci}\Ga(x,y)\sigma_{\ngi}(dy)&=\int_{{\ci_1}}\Ga(x,y)\sigma_{\ngi_1}(dy)+\int_{{\ci_2}}\Ga(x,y)\sigma_{\ngi_2}(dy)\\
&\quad -\int_{{\ci_1}\cap {\ci_2}}\Ga(x,y)\sigma_{\ngi}^+(dy)+\int_{\ci\setminus({\ci_1}\cup {\ci_2})}\Ga(x,y)\sigma_{\ngi}(dy) \notag 
\end{align}
where $\sigma_{\ngi}^+$ is the positive part of $\sigma_{\ngi}$. Considering the same decomposition as in \eqref{eq:splitintegralGa} with $\psia(y)$, instead of $\Ga(x,y)$, and also considering: equation \eqref{eq:intPsiaMenor0}; $\int_{{\ci_1}}\psia(y)\sigma_{\ngi_1}(dy)=0$; and  $\int_{{\ci_2}}\psia(y)\sigma_{\ngi_2}(dy)=0$, we obtain
\begin{equation}
-\int_{{\ci_1}\cap {\ci_2}}\psia(y)\sigma_{\ngi}^+(dy)+\int_{\ci\setminus({\ci_1}\cup {\ci_2})}\psia(y)\sigma_{\ngi}(dy)\leq 0.
\end{equation}
For every $x$ consider $k(x)\geq 0$ such that $k(x)\psia({\bar{b}_2})=\Ga(x,{\bar{b}_2})$.
We have $k(x)\psia({\bar{b}_2})\leq \Ga(x,{\bar{b}_2})$ for $y\leq {\bar{b}_2}$ and $k(x)\psia({\bar{b}_2})\geq \Ga(x,{\bar{b}_2})$ for $y\geq {\bar{b}_2}$ and therefore
\begin{align*}
&-\int_{{\ci_1}\cap {\ci_2}}\Ga(x,y)\sigma_{\ngi}^+(dy)
+\int_{\ci\setminus({\ci_1}\cup {\ci_2})}\Ga(x,y)\sigma_{\ngi}(dy) \\ 
&\quad = k(x) \left(-\int_{{\ci_1}\cap {\ci_2}}\psia(y)\sigma_{\ngi}^+(dy)
+\int_{\ci\setminus({\ci_1}\cup {\ci_2})}\psia(y)\sigma_{\ngi}(dy)\right) \leq 0.
\end{align*}
The first two terms on the right-hand side of equation \eqref{eq:splitintegralGa} are also non-positive, and we conclude that \eqref{iv} in Condition \ref{cond:continuation} holds. To prove (\ref{ii}) we consider the decomposition in \eqref{eq:splitintegralGa} with $\phia(y)$ instead of $\Ga(x,y)$ and $k\geq 0$ such that $k\psia({\bar{b}_2})=\phia({\bar{b}_2})$; the same considerations done to prove \eqref{iv} conclude the result in this case.
\end{proof}
\begin{lemma}[Right merge]\label{lem:casoder}
Under the assumptions of Theorem \ref{theorem:main}, consider $\ngi_1=(a_1,b_1)$, $\ngi_2=(a_2,r)$ such that: $b_1<a_2$; and $(\desc-\igen)g(x)\geq 0$ for $x$ in $(b_1,a_2)$. Let ${\ci_1}=({\bar{a}_1},{\bar{b}_1})$ and ${\ci_2}=({\bar{a}_2},r)$ intervals such that: ${\bar{a}_1}>\ell$; ${\bar{b}_1}<r$; and ${\bar{a}_2}>\ell$. Suppose that the two pairs of intervals $(\ngi_1,{\ci_1})$, $(\ngi_2,{\ci_2})$ satisfy Condition \ref{cond:continuation}. If ${\ci_1}\cap {\ci_2}\neq \emptyset$ then, considering $\ngi=(a_1,r)$, there exists $\bar{a}$ such that $(\ngi,\ci=(\bar{a},r))$ satisfies Condition \ref{cond:continuation}.
\end{lemma}
\begin{proof}
Analogous to the previous lemma.
\end{proof}

\begin{lemma}[Total merge] \label{lem:caso2lados}
Under the assumptions of Theorem \ref{theorem:main}, consider $\ngi_1=(\ell,b_1)$, $\ngi_2=(a_2,r)$ such that: $b_1<a_2$; and $(\desc-\igen)g(x)\geq 0$ for $x$ in $(b_1,a_2)$. Let ${\ci_1}=(\ell,{\bar{b}_1})$ and ${\ci_2}=({\bar{a}_2},r)$ intervals such that the two pairs of intervals $(\ngi_1,{\ci_1})$, $(\ngi_2,{\ci_2})$ satisfy Condition \ref{cond:continuation}. If ${\ci_1}\cap {\ci_2}\neq \emptyset$ then for all $x\in \I$,
\begin{equation}\label{eq:total}
\int_{\I}\Ga(x,y)\sigma(dy)\leq 0.
\end{equation}
\end{lemma}
In consequence, the pair $(\I,\I)$ satisfies Condition \ref{cond:continuation}.  
\begin{proof}
Consider the following decomposition of the integral
\begin{align*}
\int_{\I}\Ga(x,y)\sigma(dy)&=\int_{{\ci_1}}\Ga(x,y)\sigma_{\ngi_1}(dy)+\int_{{\ci_2}}\Ga(x,y)\sigma_{\ngi_2}(dy)\\
&\qquad -\int_{{\ci_1}\cap {\ci_2}} \Ga(x,y)\sigma^+(dy).
\end{align*}
Observing that the three terms on the right-hand side are non-positive, the lemma is proved.
\end{proof}
\begin{remark}\label{remark:total}
Observe that this case can not happen under our hypothesis:
as the inversion formula \eqref{eq:inversionformula} holds, and we assume $g$ non negative,
condition \eqref{eq:total} gives a contradiction (unless $g\equiv 0$).
\end{remark}
\subsection{Proof of Theorem \ref{theorem:main}}

\begin{proof}
We first apply Algorithm \ref{algor} departing from $\negative=\cup_{i=1}^n\ngi_i$  to obtain a set of pairwise disjoint intervals 
$\{\ci_1,\ldots,\ci_m\}$. 
Observe that, under our hypothesis, the algorithm does not give as a result $C_1=\I$ (see Remark \ref{remark:total}). 
Denote $\continuation=\cup_{i=1}^m\ci_i$ and $\stopping=\I\setminus\continuation$.
Condition $\negative\subset\continuation$ follows by construction, as the algorithm enlarges the negative set. 
Furthermore, the intervals $\ci_i$ that result from the algorithm satisfy conditions \eqref{a} and \eqref{b} 
as they satisfy \eqref{i} and \eqref{ii} in Condition \ref{cond:continuation}, due to the fact that $\sigma_N=\sigma$ restricted to 
each final $\ci_i$ resulting from the algorithm.
Condition \eqref{c} is also satisfied, due to this same fact.
It remains to prove that this is in fact the continuation region associated with the optimal stopping problem. 
We use the Dynkin's characterization as the minimal $\desc$-excessive majorant to prove that
\begin{equation*}
\Va(x):=\int_{\I\setminus \continuation}\Ga(x,y)\sigma(dy)
\end{equation*}
is the value function. Since $\sigma(dy)$ is non-negative in $\I\setminus \continuation$ we have that $\Va$ is $\desc$-excessive. 
For $x\in \I$, we have
\begin{equation}
\label{eq:gVa}
g(x)=\int_{\I}\Ga(x,y)\sigma(dy)=\Va(x) + \sum_{i=1}^m \int_{\ci_i}\Ga(x,y)\sigma(dy). 
\end{equation}
If $x\notin\continuation$, the sum in the r.h.s. of \eqref{eq:gVa} vanishes by \eqref{a} and \eqref{b}, as in the proof of Lemma \ref{lem:waeqg}. This gives $V(x)=g(x)$ for $x\in\stopping$.
On the other hand, based on \eqref{c}, this same sum is non-positive if $x\in\continuation$.
This gives $V(x)\geq g(x)$ for $x\in\continuation$.
%
%
%
%
We have then proved that $\Va$ is a  majorant of $g$. We have, up to now, $\Va(x)\geq \sup_{\tau}\E_x\left(\ea{\tau}g(X_\tau)\right)$. Finally observe that, denoting by $\stopping$ the set $\I\setminus\continuation$
$$\Va(x)=\E_x\left(\ea{\hit{S}}\Va(X_{\hit{\stopping}})\right)=\E_x \left(\ea{\hit{\stopping}}g(X_{\hit{\stopping}})\right),$$
where the first equality is a consequence of Lemma \ref{lem:Wa}. We conclude that $\Va$ is the value function and that $\stopping$ is the stopping region, finishing the proof.
\end{proof}

\bibliographystyle{plainnat}

\end{document}